\documentclass[10pt]{amsart}
\usepackage{amssymb}
\newtheorem{lemma}{Lemma}
\newtheorem{corollary}{Corollary}
\newtheorem{theorem}{Theorem}

\newtheorem{theorem a}{Theorem A}
\newtheorem{theorem b}{Theorem B}
\newtheorem{theorem C}{Theorem C}

\begin{document}

\title []{Some families of generalized Mathieu-type power series, associated
probability distributions and related functional inequalities involving
complete monotonicity and log-convexity\\}%

\author[\v Z. Tomovski, K. Mehrez]{\v Zivorad Tomovski and Khaled Mehrez }
\address{\v{Z}ivorad Tomovski. University "St. Cyril and Methodius", Faculty
of Natural Sciences and Mathematics, Institute of Mathematics, Repubic of
Macedonia.}
\email{tomovski@pmf.ukim.edu.mk}

\address{\v{Z}ivorad Tomovski, Department of Mathematics, University of Rijeka, Radmile Matejcic 2, 51000 Rijeka, Croatia}
\email{zivorad.tomovski@math.uniri.hr}
\address{Khaled Mehrez. D\'epartement de Math\'ematiques ISSAT Kasserine,
Universit\'e de Kairouan, Tunisia.}
\address{Khaled Mehrez, D\'epartement de Math\'ematiques Facult\'e de Science de Tunis,
Universit\'e Tunis el Manar, Tunisia.}
\email{k.mehrez@yahoo.fr}
\maketitle
\begin{center}
Dedicated to Professor Neven Elezovi\'c on the occasion of his sixtieth birthday.
\end{center}

\begin{abstract}
By making use of the familiar Mathieu series and its generalizations, the
authors derive a number of new integral representations and present a
systematic study of probability density functions and probability
distributions associated with some generalizations of the Mathieu series. In
particular, the mathematical expectation, variance and the characteristic
functions, related to the probability density functions of the considered
probability distributions are derived. As a consequence, some interesting
inequalities involving complete monotonicity and log-convexity are derived.
\end{abstract}



\bigskip
\bigskip
\section{Introduction}

The following familiar infinite series 
\begin{equation}\label{(1)}
S\left( r\right) =\sum\limits_{n=1}^{\infty }{\frac{{2n}}{{\left( {%
n^{2}+r^{2}}\right) ^{2}}}}\;\;\;\left( {r\in \mathbb{R}^{+}}\right)
\end{equation}%
is named after Emile Leonard Mathieu (1835-1890), who investigated it in his
1890 work \cite{61} on elasticity of solid bodies. Integral representations of
(\ref{(1)}) is given by (see \cite{51}) 
\begin{equation}\label{(2)}
S\left( r\right) =\frac{1}{r}\int\limits_{0}^{\infty }{\frac{{t\sin \left( {%
rt}\right) }}{{e^{t}-1}}}dt
\end{equation}%
Several interesting problems and solutions dealing with integral
representations and bounds for the following slight generalization of the
Mathieu series with a fractional power 
\begin{equation}\label{(3)}
S_{\mu }\left( r\right) =\sum\limits_{n=1}^{\infty }{\frac{{2n}}{{\left( {%
n^{2}+r^{2}}\right) ^{\mu }}}}\;\;\;\left( {r\in R^{+};\mu >1}\right)
\end{equation}%
can be found in the works by Cerone and Lenard \cite{71}, Diananda \cite{81}, and
Tomovski and Trencevski \cite{22}. Motivated essentially by the works of Cerone
and Lenard \cite{71}, Srivastava and Tomovski in \cite{12} defined a family of
generalized Mathieu series 
\begin{equation}\label{(4)}
S_{\mu }^{\left( {\alpha ,\beta }\right) }\left( {r;}\text{{a}}\right)
=S_{\mu }^{\left( {\alpha ,\beta }\right) }\left( {r;\left\{ {a_{k}}\right\}
_{k=1}^{\infty }}\right) =\sum\limits_{n=1}^{\infty }{\frac{{2a_{n}^{\beta }}%
}{{\left( {a_{n}^{\alpha }+r^{2}}\right) ^{\mu }}}}\;\;\left( {r,\alpha
,\beta ,\mu \in R^{+}}\right)
\end{equation}%
where it is tacitly assumed that the positive sequence 
\begin{equation*}
\text{a}=\left\{ {a_{n}}\right\} _{n=1}^{\infty }=\left\{ {%
a_{1},a_{2},a_{3},.....}\right\} \left( \mathop {\lim a_n }%
\limits_{n\rightarrow \infty }=\infty \right)
\end{equation*}%
is so chosen that the infinite series in definition (\ref{(4)}) converges, that is,
that the following auxiliary series 
\begin{equation*}
\sum\limits_{n=1}^{\infty }{\frac{1}{{a_{n}^{\mu \alpha -\beta }}}}
\end{equation*}%
is convergent. Comparing the definitions (\ref{(1)}), (\ref{(3)}) and (\ref{(4)}), we see that $%
S_{2}\left( r\right) =S\left( r\right) $ and $S_{\mu }\left( r\right)
=S_{\mu }^{\left( {2,1}\right) }\left( {r,\left\{ n\right\} _{n=1}^{\infty }}%
\right) $. Furthermore, the special cases $S_{2}^{\left( {2,1}\right)
}\left( {r;\left\{ {a_{n}}\right\} _{n=1}^{\infty }}\right) ,$ $S_{\mu
}\left( r\right) =S_{\mu }^{\left( {2,1}\right) }\left( {r;\left\{ n\right\}
_{n=1}^{\infty }}\right) ,$ $S_{\mu }^{\left( {2,1}\right) }\left( {%
r;\left\{ n{^{\gamma }}\right\} _{n=1}^{\infty }}\right) $ and $S_{\mu
}^{\left( {\alpha ,\alpha /2}\right) }\left( {r;\left\{ n\right\}
_{n=1}^{\infty }}\right) $ were investigated by Cerone-Lenard \cite{71}, Diananda
\cite{81}; and Tomovski \cite{15}. For more details the interested reader is referred to the papers
 \cite{ 4, 7,12, ZTT,22, 15, 16, 255, 18}

In this paper we consider a power series%
\begin{eqnarray}\label{zk11}
S_{\mu ,\nu }^{\left( {\alpha ,\beta }\right) }\left( {r,}\text{a}{;z}%
\right) &=&S_{\mu ,\nu }^{\left( {\alpha ,\beta }\right) }\left( {r,\left\{ {%
a_{k}}\right\} _{k=1}^{\infty };z}\right) =\sum\limits_{n=1}^{\infty }{\frac{%
{2a_{n}^{\beta }}\left( \nu \right) _{n}z^{n}}{{\left( {a_{n}^{\alpha }+r^{2}%
}\right) ^{\mu }n!}}}\;\; \\
&&\left( {r,\alpha ,\beta ,\mu \in R^{+};}\left\vert z\right\vert \leq
1\right)
\end{eqnarray}%
We denote%
\begin{equation}
S_{\mu ,\nu }^{\left( {\alpha ,\beta }\right) }\left( {r,}\text{a}{;1}%
\right) \equiv S_{\mu ,\nu }^{\left( {\alpha ,\beta }\right) }\left( {r,}%
\text{a}\right) ,\;S_{\mu ,1}^{\left( {\alpha ,\beta }\right) }\left( {r,}%
\text{a}{;1}\right) \equiv S_{\mu }^{\left( {\alpha ,\beta }\right) }\left( {%
r,}\text{a}\right)
\end{equation}

\begin{equation}
\;S_{\mu ,\nu }^{\left( {\alpha ,\beta }\right) }\left( {r,}\text{a}{;-1}%
\right) =\widetilde{S}_{\mu ,\nu }^{\left( {\alpha ,\beta }\right) }\left( {%
r,}\text{a}\right) ,\;S_{\mu ,1}^{\left( {\alpha ,\beta }\right)
}\left( {r,}\text{a}{;-1}\right) \equiv \widetilde{S}_{\mu }^{\left( {\alpha
,\beta }\right) }\left( {r,}\text{a}\right)
\end{equation}

For $a_n=n, \alpha=2, \beta=1, \nu=1$ and $\mu$ with $\mu+1$ the series (\ref{zk11}) was introduced and considered by Tomovski and Pogany in  \cite{18}.

\bigskip
This paper is organized as follows: in Section 2, we present new integral representations for generalised Mathieu series. In particular, we present a new type integral for the Mathieu series for a special case. In section 3,  we introduce, develop and investigate
probability distribution functions (PDF) associated with the Mathieu series and their generalizations. 
As consequences some inequalities are derived. In Section 3,  we show the complete monotonicity and 
log-convexity properties for generalised Mathieu series. Moreover, as consequences of these results, we presented some functional inequalities as well as lower and upper bounds for generalised Mathieu series.
\bigskip
\section{Integral expression of Mathieu series}

Our first main results is the next theorem.

\begin{theorem}
Let $r,\alpha ,\beta ,\nu ,\mu >0$ and $\gamma \left( \mu \alpha -\beta
\right) >1$ $.$ Then the Mathieu type power series $S_{\mu ,\nu }^{\left( {%
\alpha ,\beta }\right) }\left( {r,\left\{ {k}^{\gamma }\right\}
_{k=1}^{\infty };z}\right) $ has the integral representation 
\begin{equation}
S_{\mu ,\nu }^{\left( {\alpha ,\beta }\right) }\left( {r,\left\{ {k}^{\gamma
}\right\} _{k=1}^{\infty };z}\right) =\frac{2\nu z}{\Gamma \left( \mu
\right) }\int_{0}^{\infty }\frac{t^{\gamma \left[ \mu \alpha -\beta \right]
-1}e^{-t}}{\left( 1-ze^{-t}\right) ^{\nu +1}}\text{ }_{1}\Psi _{1}\left[
\left( \mu ,1\right) ;\left( \gamma \left( \mu \alpha -\beta \right) ,\gamma
\alpha \right) ;-r^{2}t^{\gamma \alpha }\right] dt,
\end{equation}%
where ${}_{p}\Psi _{q}$ denotes the Fox-Wright generalization of the
hypergeometric ${}_{p}F_{q}$ function with $p$ numerator and $q$ denominator
parameters, defined by [\cite{SRT}, p. 50, Eq. 1.5 (21)] 
\begin{equation}
{}_{p}\Psi _{q}\left[ \left( \alpha _{1},A_{1}\right) ,...,\left( \alpha
_{p},A_{p}\right) ;\left( \beta _{1},B_{1}\right) ,...,\left( \beta
_{q},B_{q}\right) ;z\right] =\sum_{n=0}^{\infty }\frac{\prod_{j=1}^{p}\Gamma
(\alpha _{j}+nA_{j})}{\prod_{j=1}^{q}\Gamma (\beta _{j}+nB_{j})}.\frac{z^{n}%
}{n!},
\end{equation}%
with 
\begin{equation*}
\left( A_{j}\in \mathbb{R}^{+},j=1,...,p;B_{j}\in \mathbb{R}%
^{+},j=1,...,q;1+\sum_{j=1}^{q}B_{j}-\sum_{j=1}^{p}A_{j}>0\right) .
\end{equation*}%
.
\end{theorem}

\begin{proof}
First of all, we find from the definition (\ref{zk11}) that 
\begin{equation}
S_{\mu ,\nu }^{\left( {\alpha ,\beta }\right) }\left( {r,\left\{ {a_{k}}%
\right\} _{k=1}^{\infty };z}\right) =2\sum_{m=0}^{\infty }\left( \overset{%
\mu +m-1}{\underset{m}{}}\right) \left( -r^{2}\right) ^{m}\sum_{n=1}^{\infty
}\frac{\left( \nu \right) _{n}z^{n}}{a_{n}^{\left( \mu +m\right) \alpha
-\beta }n!}
\end{equation}%
So, 
\begin{equation*}
\begin{split}
S_{\mu ,\nu }^{\left( {\alpha ,\beta }\right) }\left( {r,\left\{ {k}^{\gamma
}\right\} _{k=1}^{\infty };z}\right) & =2z\sum_{m=0}^{\infty }\left( \overset%
{\mu +m-1}{\underset{m}{}}\right) \left( -r^{2}\right)
^{m}\sum_{n=0}^{\infty }\frac{\left( \nu \right) _{n+1}z^{n}}{\left(
n+1\right) ^{\left[ \left( \mu +m\right) \alpha -\beta \right] \gamma }(n+1)!} \\
& =2\nu z\sum_{m=0}^{\infty }\left( \overset{\mu +m-1}{\underset{m}{}}%
\right) \left( -r^{2}\right) ^{m}\sum_{n=0}^{\infty }\frac{\left( \nu
+1\right) _{n}z^{n}}{\left( n+1\right) ^{\left[ \left( \mu +m\right) \alpha
-\beta \right] \gamma +1}n!} \\
& =2\nu z\sum_{m=0}^{\infty }\left( \overset{\mu +m-1}{\underset{m}{}}%
\right) \left( -r^{2}\right) ^{m}\Phi _{\nu +1}^{\ast }\left( z,\left[
\left( \mu +m\right) \alpha -\beta \right] \gamma +1,1\right) .
\end{split}%
\end{equation*}%
Here $\Phi _{\nu +1}^{\ast }$denotes a Hurwitz-Lerch Zeta function, defined
by Lin and Srivastava \cite{Li} 
\begin{equation*}
\Phi _{\nu }^{\ast }\left( z,s,a\right) =\sum_{n=0}^{\infty }\frac{\left(
\nu \right) _{n}}{n!}\frac{z^{n}}{\left( n+a\right) ^{s}}
\end{equation*}%
\begin{equation*}
\left( \nu \in \mathbb{C};a\in \mathbb{C}\backslash \mathbb{Z}_{0}^{-},\text{
}s\in \mathbb{C},\text{ }\left\vert z\right\vert <1;\Re \left( s-\nu \right)
>1\text{ when }\left\vert z\right\vert =1\right) .
\end{equation*}%
A special case of $\Phi _{\nu }^{\ast }$ with $\nu =1$ give us the
well-known Hurwitz-Lerch Zeta function  
\begin{equation*}
\Phi \left( z,s,a\right) =\sum_{n=0}^{\infty }\frac{z^{n}}{\left( n+a\right)
^{s}}
\end{equation*}

\begin{equation*}
\left( a\in \mathbb{C}\backslash \mathbb{Z}_{0}^{-},\text{ }s\in \mathbb{C}%
,\left\vert z\right\vert <1;\Re \left( s\right) >1\text{ when }\left\vert
z\right\vert =1\right) .
\end{equation*}%
Now, by making use of the familiar integral representation (see \cite{Li})%
\begin{equation*}
\Phi _{\nu }^{\ast }\left( z,s,a\right) =\frac{1}{\Gamma \left( s\right) }%
\int_{0}^{\infty }\frac{t^{s-1}e^{-at}}{\left( 1-ze^{-t}\right) ^{\nu }}dt
\end{equation*}%
\begin{equation*}
\left( \Re \left( a\right) >0;\Re \left( s\right) >0\text{ when }\left\vert
z\right\vert \leq 1,z\neq 1;\Re \left( s\right) >1\text{ when }z=1\right)
\end{equation*}%
we get%
\begin{equation*}
\begin{split}
S_{\mu ,\nu }^{\left( {\alpha ,\beta }\right) }\left( {r,\left\{ {k}^{\gamma
}\right\} _{k=1}^{\infty };z}\right) & =2\nu z\sum_{m=0}^{\infty }\frac{%
\binom{\mu +m-1}{m}}{\Gamma \left( \gamma \left[ \left( \mu +m\right) \alpha
-\beta \right] \right) }\left( -r^{2}\right) ^{m}\int_{0}^{\infty }\frac{%
t^{\gamma \left[ \left( \mu +m\right) \alpha -\beta \right] -1}e^{-t}}{%
\left( 1-ze^{-t}\right) ^{\nu +1}}dt \\
& =2\nu z\int_{0}^{\infty }\frac{t^{\gamma \left[ \mu \alpha -\beta \right]
-1}e^{-t}}{\left( 1-ze^{-t}\right) ^{\nu +1}}\left[ \sum_{m=0}^{\infty }%
\frac{\binom{\mu +m-1}{m}}{\Gamma \left( \gamma \left[ \left( \mu +m\right)
\alpha -\beta \right] \right) }\left( -r^{2}t^{\gamma \alpha }\right) ^{m}%
\right] dt \\
& =\frac{2\nu z}{\Gamma \left( \mu \right) }\int_{0}^{\infty }\frac{%
t^{\gamma \left[ \mu \alpha -\beta \right] -1}e^{-t}}{\left(
1-ze^{-t}\right) ^{\nu +1}}\text{ }_{1}\Psi _{1}\left[ \left( \mu ,1\right)
;\left( \gamma \left( \mu \alpha -\beta \right) ,\gamma \alpha \right)
;-r^{2}t^{\gamma \alpha }\right] dt.
\end{split}%
\end{equation*}%
This ends the proof.
\end{proof}
\newpage
\textsc{\textbf{Concluding Remarks:}}\newline
\noindent 1. As the convergence interval of $S_{\mu ,\nu }^{\left( {\alpha
,\beta }\right) }\left( {r,\left\{ {k}^{\gamma }\right\} _{k=1}^{\infty };z}%
\right) $ is $\left[ -1,1\right] ,$ we easily conclude the following
representations 
\begin{equation}
\;S_{\mu ,\nu }^{\left( {\alpha ,\beta }\right) }\left( {r,}\left\{
k^{\gamma }\right\} \right) \equiv \;\frac{2\nu }{\Gamma \left( \mu \right) }%
\int_{0}^{\infty }\frac{t^{\gamma \left[ \mu \alpha -\beta \right] -1}e^{-t}%
}{\left( 1-e^{-t}\right) ^{\nu +1}}\text{ }_{1}\Psi _{1}\left[ \left( \mu
,1\right) ;\left( \gamma \left( \mu \alpha -\beta \right) ,\gamma \alpha
\right) ;-r^{2}t^{\gamma \alpha }\right] dt,
\end{equation}%
and 
\begin{equation}
\;\widetilde{S}_{\mu ,\nu }^{\left( {\alpha ,\beta }\right) }\left( {r,}%
\left\{ k^{\gamma }\right\} \right) \equiv \;\frac{2\nu }{\Gamma \left( \mu
\right) }\int_{0}^{\infty }\frac{t^{\gamma \left[ \mu \alpha -\beta \right]
-1}e^{-t}}{\left( 1+e^{-t}\right) ^{\nu +1}}\text{ }_{1}\Psi _{1}\left[
\left( \mu ,1\right) ;\left( \gamma \left( \mu \alpha -\beta \right) ,\gamma
\alpha \right) ;-r^{2}t^{\gamma \alpha }\right] dt.
\end{equation}%
\noindent 2. Specially, for $z=e^{-x}<1,$ we find that%
\begin{equation}
S_{\mu ,\nu }^{\left( {\alpha ,\beta }\right) }\left( {r,\left\{ {k}^{\gamma
}\right\} _{k=1}^{\infty };e}^{-x}\right) =\frac{2\nu e^{-x}}{\Gamma \left(
\mu \right) }\int_{0}^{\infty }\frac{t^{\gamma \left( \mu \alpha -\beta
\right) -1}e^{-t}}{\left( 1-e^{-(x+t)}\right) ^{\nu +1}}\text{ }_{1}\Psi _{1}%
\left[ \left( \mu ,1\right) ;\left( \gamma \left( \mu \alpha -\beta \right)
,\gamma \alpha \right) ;-r^{2}t^{\gamma \alpha }\right] dt,
\end{equation}%
and 
\begin{equation}
\widetilde{S}_{\mu ,\nu }^{\left( {\alpha ,\beta }\right) }\left( {r,\left\{ 
{k}^{\gamma }\right\} _{k=1}^{\infty };e}^{-x}\right) =\frac{2\nu e^{-x}}{%
\Gamma \left( \mu \right) }\int_{0}^{\infty }\frac{t^{\gamma \left( \mu
\alpha -\beta \right) -1}e^{-t}}{\left( 1+e^{-(x+t)}\right) ^{\nu +1}}\text{ 
}_{1}\Psi _{1}\left[ \left( \mu ,1\right) ;\left( \gamma \left( \mu \alpha
-\beta \right) ,\gamma \alpha \right) ;-r^{2}t^{\gamma \alpha }\right] dt.
\end{equation}%
\noindent 3. In a similar manner, we get 
\begin{equation*}
\begin{split}
S_{\mu ,\nu }^{\left( {\alpha ,\beta }\right) }\left( {r,\left\{ {k}%
^{q/\alpha }\right\} _{k=1}^{\infty };z}\right) &=2\nu \sum_{m=0}^{\infty }\left( \overset{\mu +m-1}{\underset{m}{}}\right)
\left( -r^{2}\right) ^{m}\Phi _{\nu +1}^{\ast }\left( z,q\left[ \left( \mu
+m\right) -\frac{\beta }{\alpha }\right] \gamma ,1\right) \\
&=2\nu \int_{0}^{\infty }\frac{t^{q\left[ \mu -\frac{\beta }{\alpha }\right]
-1}e^{-t}}{\left( 1-ze^{-t}\right) ^{\nu +1}}\left[ \sum_{m=0}^{\infty }%
\frac{\left( \overset{\mu +m-1}{\underset{m}{}}\right) }{\Gamma \left( q%
\left[ \left( \mu +m\right) -\frac{\beta }{\alpha }\right] \right) }\left(
-r^{2}t^{q}\right) ^{m}\right] dt \\
&=\frac{2\nu }{\Gamma \left( q\left[ \mu -\frac{\beta }{\alpha }%
\right] \right) }\int_{0}^{\infty }\frac{t^{q\left[ \mu -\frac{\beta }{%
\alpha }\right] -1}e^{-t}}{\left( 1-ze^{-t}\right) ^{\nu +1}}\text{ }%
_{1}F_{q}\left( \mu ;\Delta \left( q;q\left[ \mu -\frac{\beta }{\alpha }%
\right] \right) ;-r^{2}\left( \frac{t}{q}\right) ^{q}\right) dt.
\end{split}
\end{equation*}%
\begin{equation*}
\left( r,\alpha ,\beta ,\nu \in 
\mathbb{R}
^{+};\mu -\frac{\beta }{\alpha }>q^{-1};q\in 
\mathbb{N}
\right) ,
\end{equation*}%
where for convenience, $\Delta \left( q;\lambda \right) $ abbreviates the
array of $q$ parameters%
\begin{equation*}
\frac{\lambda }{q},\frac{\lambda +1}{q},...,\frac{\lambda +q-1}{q}\text{ \ }%
\left( q\in 
\mathbb{N}
\right) .
\end{equation*}%
\noindent 4. For $q=2$, this integral representation, can easily be simplified to the
form:%
\begin{equation*}
\begin{split}
S_{\mu ,\nu }^{\left( {\alpha ,\beta }\right) }\left( {r,\left\{ {k}%
^{2/\alpha }\right\} _{k=1}^{\infty };z}\right) 
&=\frac{2\nu }{\Gamma \left( 2\left[ \mu -\frac{\beta }{\alpha }\right]
\right) }\int_{0}^{\infty }\frac{t^{2\left[ \mu -\frac{\beta }{\alpha }%
\right] -1}e^{-t}}{\left( 1-ze^{-t}\right) ^{\nu +1}}\text{ }_{1}F_{2}\left(
\mu ;\mu -\frac{\beta }{\alpha },\mu -\frac{\beta }{\alpha }+\frac{1}{2};-%
\frac{r^{2}t^{2}}{4}\right) dt,
\end{split}
\end{equation*}%
\begin{equation*}
\left( r,\alpha ,\beta ,\nu \in 
\mathbb{R}
^{+};\mu -\frac{\beta }{\alpha }>\frac{1}{2}\right) .
\end{equation*}%
\noindent 5. In a similar manner, a limit case, when $\beta \rightarrow 0$ would formally
yield the formula:%
\begin{equation*}
\begin{split}
S_{\mu ,\nu }^{\left( {\alpha ,0}\right) }\left( {r,\left\{ {k}^{2/\alpha
}\right\} _{k=1}^{\infty };z}\right) 
&=\frac{2\nu }{\Gamma \left( 2\mu \right) }\int_{0}^{\infty }\frac{t^{2\mu
-1}e^{-t}}{\left( 1-ze^{-t}\right) ^{\nu +1}}\text{ }_{0}F_{1}\left( -;,\mu +%
\frac{1}{2};-\frac{r^{2}t^{2}}{4}\right) dt
\end{split}
\end{equation*}

\begin{equation*}
\left( r,\alpha ,\nu \in 
\mathbb{R}
^{+};\mu >\frac{1}{2}\right) .
\end{equation*}%
\noindent 6. On the other hand, we have
\begin{equation}\label{116}
\begin{split}
S_{2,\nu }^{\left( 2{,1}\right) }\left( {r,\left\{ {k}\right\}
_{k=1}^{\infty };z}\right)& =\sum_{n=1}^{\infty }\frac{\left[ \left(
n+ir\right) +\left( n-ir\right) \right] \left( \nu \right) _{n}z^{n}}{\left[
\left( n+ir\right) \left( n-ir\right) \right] ^{2}n!}\\
&=\frac{1}{2ir}\left[ \Phi _{\nu }^{\ast }\left( z,2,ir\right) -\Phi _{\nu
}^{\ast }\left( z,2,-ir\right) \right]
\end{split}
\end{equation}
As a matterof fact, in terms of the Riemann-Liouville fractional derivative
operator $D_{z}^{\nu }$ , defined by 
\begin{equation*}
D_{z}^{\mu }\left\{ f\left( z\right) \right\} =\left\{ 
\begin{array}{c}
\frac{1}{\Gamma \left( -\mu \right) }\int_{0}^{z}\left( z-t\right) ^{-\mu
-1}f\left( t\right) dt\text{ \ \ }\left( \Re \left( \mu \right) <0\right) \\ 
\frac{d^{m}}{dz^{m}}\left\{ D_{z}^{\mu -m}f\left( z\right) \right\} \text{ \ 
}\left( m-1\leq \Re \left( \mu \right) <m\right) \text{ \ }\left( m\in 
\mathbb{N}
\right)%
\end{array}%
\right.
\end{equation*}%
it is easily seen from the series definitions in (\ref{zk11}) and (\ref{116}) that 
\begin{equation*}
\Phi _{\nu }^{\ast }\left( z,s,a\right) =\frac{1}{\Gamma \left( \nu \right) }%
D_{z}^{\nu -1}\left\{ z^{\nu -1}\Phi \left( z,s,a\right) \right\}
\end{equation*}%
we obtain%
\begin{equation}
S^{(2,1)}_{2,\nu}\left(r;\left\{n\right\}_{n=1}^\infty;z\right) =\frac{1}{2ir\Gamma \left( \nu
\right) }\left\{ D_{z}^{\nu -1}\left[ z^{\nu -1}\Phi \left( z,2,ir\right) %
\right] -D_{z}^{\nu -1}\left[ z^{\nu -1}\Phi \left( z,2,-ir\right) \right]
\right\},
\end{equation}
where $
0\leq Re(\nu)<1.$\\
\noindent 7. In \cite{GP}, the authors defined generalized $\beta -$Mittag-Leffler
functions 
\begin{equation}
E_{\beta ,\nu ,\gamma }(x)=\sum_{k=0}^{\infty }\frac{x^{k}}{[\Gamma (\nu
k+\gamma )]^{\beta }}.
\end{equation}%
For $\nu =\gamma =1$ and $\beta \in \mathbb{N},$ we get the hyper-Bessel
function. We define a new family of $\beta-$Mittag-Leffler functions 
\begin{equation}\label{z7}
E_{\beta ,\nu ,\gamma }^{(\tau )}(x)=\sum_{k=0}^{\infty }\frac{(\tau
)_{k}x^{k}}{k![\Gamma (\nu k+\gamma )]^{\beta }}.  
\end{equation}%
Specially, for $\tau =1,$ we obtain $E_{\beta ,\nu ,\gamma
}^{(1)}(x)=E_{\beta ,\nu ,\gamma }(x).$
The Mathieu series $S_{\mu ,\nu }^{\left( {\alpha ,\beta }\right) }\left(
r,\left\{ \Gamma (\gamma n+\delta )\right\} _{n=1}^{\infty },z\right) $
admits the following series representation:
\begin{equation*}
\begin{split}
S_{\mu ,\nu }^{\left( {\alpha ,\beta }\right) }\left( r,\left\{ \Gamma
(\gamma n+\delta )\right\} _{n=1}^{\infty },z\right) & =\sum_{m=0}^{\infty
}\left( \overset{\mu +m-1}{\underset{m}{}}\right) \left( -r^{2}\right)
^{m}\sum_{n=1}^{\infty }\frac{(\nu )_{n}z^{n}}{\left[ \Gamma (\gamma
n+\delta )\right] ^{(\mu +m)\alpha -\beta }n!} \\
& =\sum_{m=0}^{\infty }\left( \overset{\mu +m-1}{\underset{m}{}}\right)
\left( -r^{2}\right) ^{m}\left[ E_{(\mu +m)\alpha -\beta ,\gamma ,\delta
}^{(\nu )}(z)-\frac{1}{\left[ \Gamma \left( \delta \right) \right] ^{\left(
\mu +m\right) \alpha -\beta }}\right] .
\end{split}%
\end{equation*}%
\noindent 8. The Mathieu series $S_{\mu ,1 }^{\left( {\alpha ,\beta }\right) }\left(
r,\left\{ \Gamma (\gamma n+\delta )\right\} _{n=1}^{\infty },z\right) $
admits the following series representation:
\begin{equation*}
\begin{split}
S_{\mu ,1 }^{\left( {\alpha ,\beta }\right) }\left( r,\left\{ \Gamma
(\gamma n+\delta )\right\} _{n=1}^{\infty },z\right) & =\sum_{m=0}^{\infty
}\left( \overset{\mu +m-1}{\underset{m}{}}\right) \left( -r^{2}\right)
^{m}\sum_{n=1}^{\infty }\frac{z^{n}}{\left[ \Gamma (\gamma
n+\delta )\right] ^{(\mu +m)\alpha -\beta }} \\
&=z\sum_{m=0}^{\infty
}\left( \overset{\mu +m-1}{\underset{m}{}}\right) \left( -r^{2}\right)
^{m}\sum_{n=0}^{\infty }\frac{z^{n}}{\left[ \Gamma (\gamma
n+\gamma+\delta )\right] ^{(\mu +m)\alpha -\beta }} \\
& =z\sum_{m=0}^{\infty }\left( \overset{\mu +m-1}{\underset{m}{}}\right)
\left( -r^{2}\right) ^{m} E_{(\mu +m)\alpha -\beta ,\gamma ,\gamma+\delta
}(z).
\end{split}%
\end{equation*}

In the next Theorem we give an integral expression of Mathieu series $%
S^{(\alpha,\beta)}_{3/2,1}(r;\mathbf{a};1)$, by using the Fourier transform.

\begin{theorem}
Let $r,\alpha,\beta>0$ and $\mathbf{a}=(a_k)_{k\geq1}$ be a sequences such
that the function 
\begin{equation}
f_r^{(\alpha,\beta)}(t)=\sum_{n=1}^\infty\frac{2(a_n^\alpha t^2-r^2)}{a_n^{%
\frac{\alpha}{2}-\beta} (a_n^\alpha t^2+r^2)^2}
\end{equation}
converges for all $t>0.$ Then we have 
\begin{equation*}
S^{(\alpha,\beta)}_{3/2,1}(r;\mathbf{a},1)=\frac{2}{\pi}\int_1^\infty t\sqrt{%
t^2-1}f_r^{(\alpha,\beta)}(t)dt.
\end{equation*}
\end{theorem}

\begin{proof}
In [\cite{WE}, Eq. 6.6], the author give the Fourier transform of the
function $\varphi_{c,\mu}(x)=\frac{1}{(c^2+x^2)^\mu}$ with $c>0$ and $%
\mu>1/2 $: 
\begin{equation}  \label{01}
\mathcal{F}\varphi_{c,\mu}(\xi)=\frac{2^{1-\mu}}{\Gamma(\mu)c^{\mu-1/2}}%
\left|\xi\right|^{\mu-1/2 }K_{\mu-1/2}(c\left|\xi\right|),
\end{equation}
wher $K_\alpha$ is the modified Bessel function of the second kind. On the
other hand, using the representation integral (see for example \cite{BE},
Theorem 4.17 ) 
\begin{equation}
\displaystyle K_\alpha (x)=\frac{1}{\sqrt{\pi}\Gamma(\alpha+\frac{1}{2})}%
\left(\frac{x}{2}\right)^\alpha \int_1^{+\infty} e^{-x t}
\left(t^2-1\right)^{\alpha-\frac{1}{2}}dt , \quad \alpha >- \frac{1}{2}%
,\quad x>0.
\end{equation}
and Fubini's Theorem that the even function $x^\alpha K_\alpha(x)$ belong to 
$L^1([0,\infty))$. So, by using the inversion formula and (\ref{01}) we
deduce that 
\begin{equation}  \label{02}
\begin{split}
\varphi_{c,\mu}(x)&=\frac{2^{1-\mu}}{\sqrt{2\pi}\;\Gamma(\mu)\;c^{\mu-1/2}}%
\mathcal{F}\left(\left|\xi\right|^{\mu-1/2
}K_{\mu-1/2}(c\left|\xi\right|\right)(x), \\
=&\frac{2^{2-\mu}}{\sqrt{2\pi}\;\Gamma(\mu)\;c^{\mu-1/2}}\int_0^\infty\cos(x%
\xi)\xi^{\mu-1/2 }K_{\mu-1/2}(c\xi)d\xi
\end{split}%
\end{equation}
In view of the representation integral ( see \cite{WE}, Def.5.10) 
\begin{equation}
K_\mu(z)=\int_0^\infty e^{-z\cosh(t)}\cosh(\mu t)dt,
\end{equation}
and (\ref{02}) we obtain 
\begin{equation}  \label{03}
\varphi_{c,\mu}(x)=\frac{2^{2-\mu}}{\sqrt{2\pi}\;\Gamma(\mu)\;c^{\mu-1/2}}%
\int_0^\infty\cosh((\mu-1/2)t)\ \left[\int_0^\infty\cos(x\xi)\xi^{%
\mu-1/2}e^{-c\xi\cosh(t)}d\xi\right]dt.
\end{equation}
Now, let $\mu=3/2.$ Combining the following formula (see \cite{DE}, Eq. 2.
p.391) 
\begin{equation}
\int_0^\infty\cos(x\xi)\xi e^{-c\cosh(t)\xi}d\xi=\frac{c^2\cosh(t)^2-x^2}{%
(c^2\cosh(t)^2+x^2)^2}
\end{equation}
and (\ref{03}) we obtain 
\begin{equation}\label{rrr1}
\begin{split}
\varphi_{c,3/2}(x)&=\frac{1}{\sqrt{\pi}\;\Gamma(3/2)\;c}\int_0^\infty\cosh(t)%
\frac{(c^2\cosh(t)^2-x^2)}{(c^2\cosh(t)^2+x^2)^2}dt \\
&=\frac{1}{\sqrt{\pi}\;\Gamma(3/2)\;c}\int_1^\infty t\sqrt{t^2-1}\frac{%
(c^2t^2-x^2)}{(c^2t^2+x^2)^2}dt.
\end{split}%
\end{equation}
Letting $c=a_n^\frac{\alpha}{2}$ and $x=r$ in the (\ref{rrr1}), we get 
$$\frac{2a_n^\beta}{(a_n^\alpha+r^2)^2}=\frac{2}{\pi}\int_1^\infty t\sqrt{t^2-1}\frac{%
2a_n^\beta(a_n^\alpha t^2-r^2)}{a_n^{\alpha/2}(a_n^\alpha t^2+r^2)^2}dt.$$
The interchanging between integral and summation gives the desired result.
\end{proof}
\section{Mathieu probability distribution}

The main objective of this section is to introduce, develop and investigate
probability distribution functions (PDF) associated with the Mathieu series
and their generalizations. We define a discrete random variable $X$ defined
on some fixed standard probability space $\left( \Omega ,\digamma ,P\right) $
possesing a Mathieu distribution with parameter $r>0,$ \ \ \ 
\begin{equation}
P_{\mu ,\nu }^{\left( \alpha ,\alpha \right) }\left( n,r\right) =P\left(
X=n\right) ={\frac{{2n^{\alpha }}\left( \nu \right) _{n}}{{\left( n{^{\alpha
}+r^{2}}\right) ^{\mu }n!}}}\frac{1}{S_{\mu ,\nu }^{\left( \alpha ,\alpha
\right) }\left( r\right) }\;\;
\end{equation}%
\ 
\begin{equation*}
\left( {r,\alpha ,\mu ,\nu \in R^{+}}\right) ,\text{ }n=1,2,...,
\end{equation*}%
where \ $S_{\mu ,\nu }^{\left( \alpha ,\alpha \right) }\left( r\right)
=S_{\mu ,\nu }^{\left( \alpha ,\alpha \right) }\left( {r,\left\{ n\right\}
_{n=1}^{\infty }}\right) .$ $P_{\mu ,\nu }^{\left( \alpha ,\alpha \right)
}\left( n,r\right) $ is normalized, since 
\begin{equation*}
\sum\limits_{n=1}^{\infty }P_{\mu ,\nu }^{\left( \alpha ,\alpha \right)
}\left( n,r\right) =\sum_{n=1}^{\infty }{\frac{{2n}^{\alpha }\left( \nu
\right) _{n}}{{\left( n{^{\alpha }+r^{2}}\right) ^{\mu }n!}}}\frac{1}{S_{\mu
,\nu }^{\left( \alpha ,\alpha \right) }\left( r\right) }=\frac{S_{\mu ,\nu
}^{\left( \alpha ,\alpha \right) }\left( r\right) }{S_{\mu ,\nu }^{\left(
\alpha ,\alpha \right) }\left( r\right) }=1.
\end{equation*}

\begin{theorem}
The expected value $EX$ of a Mathieu distribution $P_{\mu ,\nu }^{\left(
\alpha ,\alpha \right) }\left( n,r\right) $ is $\frac{S_{\mu,\nu
}^{\left( \alpha ,\alpha+1\right) }\left( r\right) }{S_{\mu ,\nu }^{\left( \alpha ,\alpha
\right) }\left( r\right) }$ and variance $Var\left( X\right) $ is given by%
\begin{equation*}
\frac{ S_{\mu ,\nu }^{\left( \alpha ,\alpha+2\right) }\left( r\right)
S_{\mu ,\nu }^{\left( \alpha ,\alpha\right) }\left( r\right)   -\left[ S_{\mu ,\nu
}^{\left( \alpha ,\alpha+1\right) }\left( r\right)\right]^2}{\left( S_{\mu ,\nu }^{\left(
\alpha ,\alpha \right) }\left( r\right) \right) ^{2}}.
\end{equation*}%
Moreover the following Tur\'an type inequality holds true%
\begin{equation}\label{!!!}
S_{\mu ,\nu }^{\left( \alpha ,\alpha+2\right) }\left( r\right)
S_{\mu ,\nu }^{\left( \alpha ,\alpha\right) }\left( r\right)\geq\left( S_{\mu ,\nu }^{\left(
\alpha ,\alpha +1\right) }\left( r\right) \right) ^{2}.
\end{equation}%

\begin{proof}
By computation we have 
\begin{equation*}
\begin{split}
EX &=\sum_{n=1}^{\infty }nP\left( X=n\right)\\&
 =\sum\limits_{n=1}^{\infty }{%
\frac{{2n^{\alpha +1}}\left( \nu \right) _{n}}{{\left( n{^{\alpha }+r^{2}}%
\right) ^{\mu }n!}}}\frac{1}{S_{\mu ,\nu }^{\left( \alpha ,\alpha \right)
}\left( r\right) } \\
&=\frac{S_{\mu ,\nu }^{\left( \alpha ,\alpha+1 \right)
}\left( r\right) }{S_{\mu ,\nu }^{\left( \alpha ,\alpha \right)
}\left( r\right) },
\end{split}%
\end{equation*}%
and 
\begin{equation*}
\begin{split}
EX^{2} &=\sum\limits_{n=1}^{\infty }n^{2}P\left( X=n\right) \\
&=\sum\limits_{n=1}^{\infty }{\frac{{2n^{\alpha +2}}\left( \nu \right) _{n}}{%
{\left( n{^{\alpha }+r^{2}}\right) ^{\mu }n!}}}\frac{1}{S_{\mu ,\nu
}^{\left( \alpha ,\alpha \right) }\left( r\right) } \\
&=\frac{S_{\mu ,\nu }^{\left( \alpha ,\alpha+2 \right)
}\left( r\right) }{S_{\mu ,\nu }^{\left( \alpha ,\alpha \right)
}\left( r\right) }.
\end{split}%
\end{equation*}%
Thus 
\begin{equation*}
\begin{split}
Var\left( X\right) &=EX^{2}-\left( EX\right) ^{2} \\
&=\frac{ S_{\mu ,\nu }^{\left( \alpha ,\alpha+2\right) }\left( r\right)
S_{\mu ,\nu }^{\left( \alpha ,\alpha\right) }\left( r\right)   -\left[ S_{\mu ,\nu
}^{\left( \alpha ,\alpha+1\right) }\right]^2}{\left( S_{\mu ,\nu }^{\left(
\alpha ,\alpha \right) }\left( r\right) \right) ^{2}}.
\end{split}%
\end{equation*}
Using the fact that $Var(x)$ is nonnegative, we easy get that the Tur\'an type
 inequality (\ref{!!!}) holds true.
\end{proof}
\end{theorem}

\begin{theorem}
\label{t4} \bigskip The characteristic funtion of the Mathieu distribution $%
P_{\mu +1,\nu }^{\left( 2,1\right) }\left( n,r\right) $ , $r>0$ is given by%
\begin{equation}
f_{x}\left( t\right) =\frac{\sqrt{\pi }}{\left( 2r\right) ^{\mu -1/2}\Gamma
\left( \mu +1\right) S_{\mu ,\nu }\left( r\right) }\int\limits_{0}^{\infty }%
\frac{e^{u\nu }u^{\mu +1/2}}{\left( e^{u}-e^{it}\right) ^{\nu }}J_{\mu
-1/2}\left( ru\right) du\text{ \ \ \ }\left( Re\left( \mu \right) >-\frac{1}{%
2}\right) ,
\end{equation}
where $J_{\mu-1/2}(.)$ is the Bessel function. 
\end{theorem}

\begin{proof}
Using the formula$:$%
\begin{eqnarray*}
\frac{2n}{\left( n^{2}+r^{2}\right) ^{\mu +1}} &=&\frac{\sqrt{\pi }}{\left(
2r\right) ^{\mu -1/2}}\int\limits_{0}^{\infty }e^{-nt}t^{\mu +1/2}J_{\mu
-1/2}\left( rt\right) dt \\
&&\left( \Re \left( \mu \right) >-\frac{1}{2}\right) \text{\ }
\end{eqnarray*}%
we obtain%
\begin{equation*}
\begin{split}
f_{x}\left( t\right) &=\sum\limits_{n=1}^{\infty }e^{itn}\frac{{2n}\left(
\nu \right) _{n}}{{\left( {n^{2}+r^{2}}\right) ^{\mu +1}n!}}\frac{1}{S_{\mu
,\nu }\left( r\right) } \\
&=\frac{\sqrt{\pi }}{\left( 2r\right) ^{\mu -1/2}\Gamma \left( \mu +1\right)
S_{\mu ,\nu }\left( r\right) }\int\limits_{0}^{\infty }u^{\mu +1/2}J_{\mu
-1/2}\left( ru\right) \left[ \sum_{n=1}^{\infty }\frac{\left( \nu \right)
_{n}}{n!}\left( e^{-u}e^{it}\right) ^{n}\right] du \\
&=\frac{\sqrt{\pi }}{\left( 2r\right) ^{\mu -1/2}\Gamma \left( \mu +1\right)
S_{\mu ,\nu }\left( r\right) }\int\limits_{0}^{\infty }\frac{e^{u\nu }u^{\mu
+1/2}}{\left( e^{u}-e^{it}\right) ^{\nu }}J_{\mu -1/2}\left( ru\right) du.
\end{split}%
\end{equation*}
So, the proof of Theorem \ref{t4} is complete.
\end{proof}

\begin{corollary}\label{cor}
Let $\nu>0$ and $\mu>1.$ Then the following inequality holds 
\begin{equation}  \label{zzkk}
S_{\mu,\nu }^{\left( 2,1\right) }\left( r,\left\{ n\right\} \right)
S_{\mu ,\nu }^{\left( 2,3\right) }\left( r,\left\{ n\right\} \right) \geq
\left[S_{\mu,\nu
}^{\left( 2,2\right) }\left( r,\left\{ n\right\} \right) \right]^2.
\end{equation}%
In particular,
\begin{equation*} 
S_{2,\nu }^{\left( 2,1\right) }\left( r,\left\{ n\right\} \right)
S_{2 ,\nu }^{\left( 2,3\right) }\left( r,\left\{ n\right\} \right) \geq
\left[S_{2,\nu
}^{\left( 2,2\right) }\left( r,\left\{ n\right\} \right) \right]^2.
\end{equation*}%
\end{corollary}
\begin{proof}
We consider a special case
\begin{equation}
P_{\mu ,\nu }^{\left( 2,1\right) }\left( n,r\right) =P\left( Y=n\right) ={%
\frac{{2n}\left( \nu \right) _{n}}{{\left( n{^{2}+r^{2}}\right) ^{\mu }n!}}}%
\frac{1}{S_{\mu ,\nu }^{\left( 2,1\right) }\left( r\right) }\;\;
\end{equation}%
\begin{equation*}
\left( {r,\mu ,\nu \in R^{+}}\right) ,\text{ }n=1,2,...
\end{equation*}%
Then%
\begin{equation*}
EY=\frac{S_{\mu ,\nu }^{\left( 2,2\right) }\left( r,\left\{ n\right\}
\right) }{S_{\mu ,\nu }^{\left( 2,1\right) }\left( r,\left\{ n\right\}
\right) }
\end{equation*}%
\begin{equation*}
EY^{2}=\frac{S_{\mu ,\nu }^{\left( 2,3\right) }\left( r,\left\{ n\right\}
\right) }{S_{\mu ,\nu }^{\left( 2,1\right) }\left( r,\left\{ n\right\}
\right) }
\end{equation*}%
Applying the elementary inequality $EY^{2}\geq \left( EY\right) ^{2},$we
obtain%
\begin{equation*}
S_{\mu,\nu }^{\left( 2,1\right) }\left( r,\left\{ n\right\} \right)
S_{\mu ,\nu }^{\left( 2,3\right) }\left( r,\left\{ n\right\} \right) \geq
\left[S_{\mu,\nu
}^{\left( 2,2\right) }\left( r,\left\{ n\right\} \right) \right]^2.
\end{equation*}%
\end{proof}

\begin{theorem}
Let $\nu\geq1.$ Then the following inequality
\begin{equation} \label{MM} 
S_{1,\nu }^{\left( 2,1\right) }\left( r,\left\{ n\right\} \right) S_{2,\nu
}^{\left( 2,1\right) }\left( r,\left\{ n\right\} \right) \geq 2r^2
S_{3,1}^{\left( 2,1\right) }\left( r,\left\{ n\right\} \right)+ \left[S_{2,1
}^{\left( 2,2\right) }\left( r,\left\{ n\right\} \right) \right]^2
\end{equation}
holds true.
\end{theorem}

\begin{proof}
From Corollary \ref{cor}, we have
\begin{equation}
\left[S_{1,\nu }^{\left( 2,1\right) }\left( r,\left\{ n\right\} \right) -r^{2}
S_{2,\nu }^{\left( 2,1\right) }\left( r,\left\{ n\right\} \right) \right]S_{2,\nu
}^{\left( 2,1\right) }\left( r,\left\{ n\right\} \right)\geq\left[S_{2,\nu
}^{\left( 2,2\right) }\left( r,\left\{ n\right\} \right) \right]^2.
\end{equation}
On the other hand, using the fact that the function $\nu\mapsto(\nu)_n$ is
increasing on $[1,\infty)$, we deduce that 
\begin{equation}  \label{z1k2}
\left[ S_{2,\nu }^{\left( 2,1\right) }\left( r,\left\{ n\right\} \right) %
\right] ^{2}\geq\left[ S_{2,1 }^{\left( 2,1\right) }\left( r,\left\{
n\right\} \right) \right] ^{2},
\end{equation}
and
\begin{equation}  \label{z1k3}
\left[ S_{2,\nu }^{\left( 2,2\right) }\left( r,\left\{ n\right\} \right) %
\right] ^{2}\geq\left[ S_{2,1 }^{\left( 2,2\right) }\left( r,\left\{
n\right\} \right) \right] ^{2},
\end{equation}
In \cite{Wi}, Wilkins proved the following inequality 
\begin{equation}  \label{z1k1}
\left[\sum_{n=1}^\infty\frac{n}{(n^2+r^2)^2}\right]^2\geq \sum_{n=1}^\infty%
\frac{n}{(n^2+r^2)^3}.
\end{equation}
Combining (\ref{z1k2}) and (\ref{z1k1}) we get 
\begin{equation}\label{MMM}
\left[S_{2,\nu }^{\left( 2,1\right) }\left( r,\left\{ n\right\} \right) %
\right] ^{2}\geq 2S_{3,1}^{\left( 2,1\right) }\left( r,\left\{ n\right\}
\right).
\end{equation}
In view of (\ref{MMM}) and (\ref{z1k3}) we get  the inequality (\ref{MM}).
\end{proof}

\section{Functional inequalities for Mathieu series}

Before we present our main results in this section, we recall some standard definitions and basic facts. A non-negative function $f$ defined on $( 0 , \infty)$ is called completely
monotonic  if it has derivatives of all orders and 
$$(-1 ) ^n f^{(n)}(x)\geq0,\;n\geq1$$
and $x > 0$ [\cite{180}, \cite{BN}, \cite{Kh}]. This inequality is known to be strict unless f is a constant.
By the celebrated Bernstein theorem, a function is completely monotonic if and only
if it is the Laplace transform of a non-negative measure [\cite{180}, Thm. 1.4].

\begin{theorem}
Let $\alpha,\beta,\nu,\mu>0$ and $0\leq z\leq1$. Then the function $r\mapsto
S^{(\alpha,\beta)}_{\mu,\nu}(\sqrt{r},\mathbf{a},z)$ is completely monotonic
and log-convex on $(0,\infty)$. In particular, for $r_1,r_2,\nu>0,$ the
following chain of inequalities: 
\begin{equation}  \label{001}
\left[S^{(\alpha,\beta)}_{\mu,\nu}\left(\sqrt{\frac{r_1+r_2}{2}},\mathbf{a}%
,z\right)\right]^2\leq S^{(\alpha,\beta)}_{\mu,\nu}(\sqrt{r_1},\mathbf{a},z)
S^{(\alpha,\beta)}_{\mu,\nu}(\sqrt{r_2},\mathbf{a},z)\leq2
\zeta_{\mu,\nu}(\alpha,\beta,z) S^{(\alpha,\beta)}_{\mu,\nu}(\sqrt{r_1+r_2},%
\mathbf{a},z)
\end{equation}
holds true if the following auxiliary series 
\begin{equation*}
\zeta_{\nu,\mu}(\alpha,\beta,z)=\sum_{n=1}^\infty\frac{(\nu)_n z^n}{n!
a_n^{\alpha\mu-\beta}}
\end{equation*}
is convergent.
\end{theorem}

\begin{proof}
As the function $r\mapsto\frac{2a_n^\beta(\nu)_n z^n}{n!(a_n^\alpha+r)^\mu}$
is completely monotonic on $(0,\infty)$ for all $\alpha,\beta,\mu,\nu>0.$
Using the fact that sums of completely monotonic functions are completely
monotonic too, we deduce that the function $r\mapsto
S^{(\alpha,\beta)}_{\mu,\nu}(\sqrt{r},\mathbf{a},z)$ is completely monotonic
and log-convex on $(0,\infty)$, since every completely monotonic function is
log-convex [see \cite{WI}, p. 167]. Thus for all $r_1,r_2>0,$ and $t\in[0,1]$
we get 
\begin{equation}
S^{(\alpha,\beta)}_{\mu,\nu}\left(\sqrt{\frac{tr_1+(1-t)r_2}{2}},\mathbf{a}%
,z\right) \leq [S^{(\alpha,\beta)}_{\mu,\nu}(\sqrt{r_1},\mathbf{a},z)]^t
[S^{(\alpha,\beta)}_{\mu,\nu}(\sqrt{r_2},\mathbf{a},z)]^{1-t}.
\end{equation}
Choosing $t=1/2$, the above inequality reduce to the first inequality in (%
\ref{001}). For the second inequality in (\ref{001}), we observe that the
function $r\mapsto S^{(\alpha,\beta)}_{\mu,\nu}(\sqrt{r},\mathbf{a},z)$ is
decreasing on $(0,\infty),$ and thus 
\begin{equation*}
S^{(\alpha,\beta)}_{\mu,\nu}(\sqrt{r},\mathbf{a},z)\leq
S^{(\alpha,\beta)}_{\mu,\nu}(0,\mathbf{a},z)=2 \sum_{n=1}^\infty\frac{%
(\nu)_nz^n}{n!a_n^{\alpha\mu-\beta}}=2\zeta_{\mu,\nu}(\alpha,\beta,z).
\end{equation*}
So, the function $r\mapsto\frac{S^{(\alpha,\beta)}_{\mu,\nu}(\sqrt{r},%
\mathbf{a},z)}{2\zeta_{\mu,\nu}(\alpha,\beta,z)}$ maps $(0,\infty)$ into $%
(0,1)$ and its completely monotonic on $(0,\infty)$. On the other hand,
according to Kimberling \cite{KI}, if a function $f$, defined on $(0,\infty)$%
, is continuous and completely monotonic on $(0,\infty)$ into $(0,1)$, then
the $\log f$ is super-additive, that is for all $x,y>0$, we have 
\begin{equation*}
\log f(x+y)\geq\log f(x)+\log f(y)\;\;\text{or}\;\;f(x+y)\geq f(x) f(y).
\end{equation*}
Therefore we conclude the second inequality in (\ref{001}).
\end{proof}

\begin{lemma}
\label{l1}\cite{AN} Let $f,g:[a,b]\rightarrow\mathbb{R}$, be two continuous
functions which are differentiable on $(a,b)$. Further, let $g\prime(x)\neq0$
on $(a,b).$ If $f^\prime/g^\prime$ is increasing (or decreasing) on $(a, b),$
then the functions 
\begin{equation*}
\frac{f(x)-f(a)}{g(x)-g(a)}\;\;\text{and}\;\;\frac{f(x)-f(b)}{g(x)-g(b)},
\end{equation*}
are also increasing (or decreasing) on $(a, b).$
\end{lemma}

\begin{corollary}
Let $\alpha,\beta,\mu>0$ and $0\leq z\leq1.$ Then the following inequality 
\begin{equation}  \label{002}
2\zeta_{\mu,\nu}(\alpha,\beta,z)e^{-\mu\frac{\zeta_{\mu+1,\nu}(\alpha,%
\beta,z)} {\zeta_{\mu,\nu}(\alpha,\beta,z)}r^2}\leq
S_{\mu,\nu}^{(\alpha,\beta)}(r,\mathbf{a},z)
\end{equation}
holds true if the following auxiliary series 
\begin{equation*}
\zeta_{\nu,\mu}(\alpha,\beta,z)=\sum_{n=1}^\infty\frac{(\nu)_n z^n}{n!
a_n^{\alpha\mu-\beta}}
\end{equation*}
is convergent.
\end{corollary}

\begin{proof}
Since the function $r\mapsto \frac{S_{\mu ,\nu }^{(\alpha ,\beta )}(\sqrt{r},%
\mathbf{a},z)}{2\zeta _{\mu ,\nu }(\alpha ,\beta ,z)}$ is log-convex on $%
(0,\infty )$ for al $\alpha ,\beta ,\nu >0,$ we obtain that $r\mapsto \frac{%
(S_{\mu ,\nu }^{(\alpha ,\beta )}(\sqrt{r},\mathbf{a},z))^{\prime }}{S_{\mu
,\nu }^{(\alpha ,\beta )}(\sqrt{r},\mathbf{a},z)}$ is increasing on $%
(0,\infty )$. Let 
\begin{equation*}
F(r)=\log \left( \frac{S_{\mu ,\nu }^{(\alpha ,\beta )}(\sqrt{r},\mathbf{a}%
,z)}{2\zeta _{\mu ,\nu }(\alpha ,\beta ,z)}\right) ,\;\text{and}\;G(r)=r.
\end{equation*}%
Then the function 
\begin{equation*}
H(r)=\frac{F(r)}{G(r)}=\frac{F(r)-F(0)}{G(r)-G(0)},
\end{equation*}%
is increasing on $(0,\infty )$, by means of Lemma \ref{l1}. So, by using the
l'Hospital's rule we get 
\begin{equation}
\log \left( \frac{S_{\mu ,\nu }^{(\alpha ,\beta )}(\sqrt{r},\mathbf{a},z)}{%
2\zeta _{\mu ,\nu }(\alpha ,\beta ,z)}\right) \geq rF^{\prime }(0)=-r\mu 
\frac{\zeta _{\mu +1}(\alpha ,\beta ,z)}{2\zeta _{\mu }(\alpha ,\beta ,z)}.
\end{equation}%
Therefore we conclude the asserted inequality (\ref{002}).
\end{proof}
\textbf{Remarks.} \noindent 1. Taking $z=\beta=\nu=1,\:\textbf{a}=(n)_{n\geq1}$ and $\alpha=2$ in (\ref{002}), we obtain the following inequality
\begin{equation}\label{re1}
2\zeta(2\mu-1)\exp\left\{-\mu\frac{\zeta(2\mu+1)}{2\mu-1}r^2\right\}\leq S_\mu(r),\;r>0,
\end{equation}
where $\zeta(.)$ denotes the Riemann zeta function defined by
$$\zeta(p)=\sum_{n=1}^\infty\frac{1}{n^p}.$$ 
\noindent 2. We note that if we choose $\mu=2$ and $\mu=3/2$ in (\ref{re1}) we obtain the following inequalities
\begin{equation}
2\zeta(3)\exp\left\{-2\frac{\zeta(5)}{\zeta(3)}r^2\right\}\leq S(r),
\end{equation}
and
\begin{equation}
\frac{\pi^2}{3}\;\exp\left\{-\frac{\pi^2r^2}{10}\right\}\leq S_{3/2}(r),
\end{equation}
holds true for all $r>0.$

\end{document}